\documentclass[12pt,reqno,twoside]{amsart}
\usepackage{amsthm}
\usepackage{amsmath}
\usepackage{amssymb}
\usepackage{epsfig}
\usepackage{color}
\usepackage{graphics}
\usepackage{url}
\usepackage{graphicx}
\usepackage{epstopdf}

\usepackage{tikz}
\usepackage{verbatim}
\usetikzlibrary{shapes,snakes}
\usetikzlibrary{shapes,backgrounds}

\def\rectangleone{(0,-2) rectangle (4,12)}
\def\rectangletwo{(1,1) rectangle (3,9)}
\def\rectanglethree{(-1,-1) rectangle (3, 3)}
\def\rectanglefour {(-1,7) rectangle (3, 11)}
\def\rectanglefive{(1,7) rectangle (5,11)}

\numberwithin{equation}{section}

\title{Ergodicity for infinite periodic translation surfaces}

\author{Pascal Hubert}

\address{LATP, case cour A, Facult\'e des sciences de Saint J\'er\^ome, 
Avenue Escadrille Normandie-Niemen, 
13397 Marseille Cedex 20, France \vskip 0.5em
 {\tt hubert@cmi.univ-mrs.fr}}

\author{Barak Weiss}

\address{Ben Gurion University, Be'er Sheva, Israel 84105
{\tt barakw@math.bgu.ac.il}}

\newif\ifdraft\drafttrue

\draftfalse

\font\sb = cmbx8 scaled \magstep0

\font\sn = cmssi8 scaled \magstep0

\long\def\combarak#1{\ifdraft{\sb #1 }\else\ignorespaces\fi}

\newcommand\name[1]{\label{#1}{\ifdraft{\sn [#1]}\else\ignorespaces\fi}}

\newcommand\eq[2]{{\ifdraft{\ \tt
[#1]}\else\ignorespaces\fi}\begin{equation}\label{eq: 
#1}{#2}\end{equation}}

\newcommand {\equ}[1]     {\eqref{eq: #1}}

\newcommand{\ov}[1]{\overline{#1}}
\newcommand{\Vis}{\mathrm{Vis}}

\newcommand{\R}{{\mathbb{R}}}

\newcommand{\HH}{{\mathbb{H}}}

\newcommand{\Z}{{\mathbb{Z}}}

\newcommand{\PPP}{{\mathbb{P}}}

\newcommand{\hol}{{\mathrm{hol}}}

\newcommand{\E}{{\mathbf{e}}}

\newcommand{\N}{{\mathbb{N}}}

\newcommand{\cl}{\overline}

\newcommand{\Ad}{{\operatorname{Ad}}}

\newcommand{\GL}{\operatorname{GL}}

\newcommand{\SL}{\operatorname{SL}}

\newcommand{\Sat}{\mathcal{O}}

\newcommand{\df}{{\, \stackrel{\mathrm{def}}{=}\, }}

\newcommand{\til}{\widetilde}

\newcommand{\sm}{\smallsetminus}

\newcommand{\vre}{\varepsilon}

\newcommand{\cc}{{\mathcal{C}}}

\font\sb = cmbx8 scaled \magstep0

\newcommand\Aff{\mathrm{Aff}}

\newcommand{\ca}{{\mathcal A}}

\newcommand {\ignore}[1]  {}

\newtheorem{thm}{Theorem}

\newtheorem{lem}[thm]{Lemma}

\newtheorem{prop}[thm]{Proposition}

\newtheorem{cor}[thm]{Corollary}

\newtheorem{claim}[thm]{Claim}

\newtheorem{Def}[thm]{Definition}

\newtheorem{remark}[thm]{Remark}

\begin{document}

\begin{abstract}
For a $\Z$-cover $\til M \to M$ of a translation surface, which
is a lattice surface, and which admits infinite strips, we
prove that almost every direction for the straightline flow is
ergodic. 
\end{abstract}

\maketitle

{\it This paper is dedicated to Howard Masur whose work is a great source of inspiration for the authors.}
\section{Introduction}
A translation surface is a topological surface equipped with a
geometric structure which makes
it possible to define a straightline flow on the surface in any
direction $\theta$. 
The study of the ergodic properties of straightline flows on compact
translation surfaces is a classical subject which has been studied for
nearly a century (see for instance \cite{KS}).    
The dynamics of the straightline flows is best understood on so-called lattice
surfaces, following celebrated work of Masur \cite{Masur-Duke} and
Veech \cite{Veech - alternative}. See \cite{MaTa, Viana, Zo}
for detailed introductions to translation surfaces, including surveys
of these and more recent developments. 

For non-compact translation surfaces, while several examples of
lattice surfaces have been given (see \cite{Ho,HuWe, HuGabi}), 
there are not many results on the dynamics of the straightline
flows. The only example which is well-understood is the infinite
staircase surface, for which all ergodic invariant Radon measures were
classified \cite{HuWe}.  
The infinite staircase is an example of a lattice surface which arises
as a $\Z$-cover of a compact translation surface, and a general
theory for such surfaces was developed in 
\cite{Zcovers}. Note that for $\Z$-covers, the ergodicity
question can be reduced to a question on the ergodicity of a
$\Z$-valued skew product over a base dynamics. In the infinite
staircase example the base dynamics is an irrational rotation of the
circle; the skew products over rotations are well-understood, and
in fact the results of \cite{HuWe} 
are essentially just a reformulation of prior work on skew products
\cite{ANSS}. However for general $\Z$-covers, one is led to the
study of skew products over interval exchange transformations. These
are poorly understood and the reduction does not essentially simplify
the analysis. Fortunately, the original
arguments of Masur \cite{Masur-Duke} can be adapted to this
situation; using this approach we solve the ergodicity question in
many cases and obtain new results about $\Z$-valued skew products over
interval exchanges. 

We now introduce the terminology needed for stating our results. 
We identify $S^1 = \R / 2\pi\Z$ with the set of
directions. For $\theta \in S^1$, throughout this paper we let
$\E_\theta \in \R^2$ denote the vector $(\cos \theta, \sin \theta)$.

Let $M$ be a compact translation surface,  and $p:\til M \to M$ a
$\Z$-cover of translation surfaces; i.e. $\til M$ is a (noncompact)
translation surface, there is a finite set $P \subset M$ such that $p:
\til M \sm p^{-1}(P) \to M \sm P$ is a covering map which is a
translation in each chart, and there is 
a translation automorphism $S: \til M \to \til M$ commuting with $p$,
such that $M$ is 
isomorphic to $\til M /\langle S \rangle$.




An {\em infinite strip} in $\til M$ is a subset isometric to an
infinite strip in $\R^2$, i.e. isometric to
$\R\times (-a,a)$ for some $a>0$.
As explained in Proposition \ref{prop: having strips} below, an
infinite strip $\Sigma \subset \til M$ projects to a cylinder $C 
\subset M$, and the lift to $\Sigma$ of a core curve $\delta$ of $C$
is not closed, so that its
endpoints are $x, \, S^kx$ for some $k = k(\Sigma) \neq 0$. The
holonomy vector of $\delta$ depends 
only on $\Sigma$ and we denote it by $v(\Sigma)$. Also we denote by
$A(\Sigma)$ the 
area of $C$. We say
that a direction $\theta$ is {\em well-approximated by
strips} if there are $\vre>0, \, k \neq 0$ and infinitely many
strips $\Sigma \subset 
\til M$ for which $k \equiv k(\Sigma)$, the
$A(\Sigma)$ are bounded below by a uniform positive bound, and 
\eq{eq: approximation}{
\left| \E_{\theta} \wedge v(\Sigma)\right| \leq (1-\vre) \,
\frac{A(\Sigma)}{2\|v(\Sigma)\|}.}

We say that $\theta$ is an ergodic direction on $\til M$
if the straightline flow in direction $\theta$ on $\til M$
is ergodic with respect to Lebesgue measure. 
We say that $\theta$ is {\em ergodic
on intermediate finite covers}, if for any $\til M \to M' \to M$, where $M'$
is a finite cover of $M$, $\theta$ is an ergodic direction on $M'$.


\begin{thm}\name{thm: main} 
Let $\til M \to M$ be a $\Z$-cover of translation
surfaces.
Suppose $\theta$ is a direction which is well-approximated by strips
in $\til M$, and ergodic on intermediate finite covers. Then $\theta$
is an ergodic direction on $\til M$.  
\end{thm}






Using Theorem \ref{thm: main} we provide examples of ergodic
directions on infinite translation surfaces:


\begin{thm}
\name{cor: Veech dichotomy}
Suppose $\til M \to M$ is a $\Z$-cover, such that $\til M$
is a lattice and has an infinite strip. Then almost every direction
on $\til M$ is ergodic; in fact there is a closed set $\Theta$
of directions, of Hausdorff dimension less than 1, such that any
$\theta \notin \Theta$ is ergodic. 
\end{thm}

\begin{remark}
\begin{enumerate}
\item
As we show at the end of the paper, there are recurrent $\Z$-covers
which do not admit infinite strips, 
and for these our methods fail. However as we show in Proposition
\ref{prop: observation strips}, these are quite rare. 
\item
The set $\Theta$ admits an explicit description in terms of the behavior 
of geodesic trajectories in
$G/\Gamma_0$, where $\Gamma_0$ is the Veech group of $\til M$. 
Namely it is the set of geodesic trajectories which do not venture
sufficiently far into the cusp 
corresponding to a cylinder in $M$ which lifts to an infinite strip in
$\til M$. 
\item
The deduction of Theorem \ref{cor: Veech dichotomy} from Theorem
\ref{thm: main} is inspired by ideas of Masur \cite{Masur-Duke}. In
his proof of unique ergodicity for translation surfaces, to prove the
existence on $M$ of long and thin rectangles in direction $\theta$,
Masur uses that a certain geodesic trajectory is not divergent {\em in the moduli space
of compact translation surfaces}. Our approach to verifying
that a.e. $\theta$ is well-approximable by cylinders is similar. To find
long and thin rectangle on $\til M$ in direction $\theta$, we study
the limiting points of the corresponding trajectory 
{\em within the $G$-orbit of $\til M$.} This hints that a {\em moduli
space of recurrent covers} might be very helpful for generalizing our
results.

\item
In the special case in which $\til M$ is the infinite staircase
surface, the ergodicity of the straightline flow in every irrational
direction follows from results of Conze \cite{Co} on cylinder flows over
irrational rotations. One can show that in this case the set $\Theta$
consists only of rational directions, thus Theorem \ref{cor: Veech
dichotomy} yields 
another proof of Conze's theorem.

\item
For compact translation surfaces, almost every direction is uniquely
ergodic. For infinite translation surfaces, one does not usually
expect unique ergodicity. In the infinite staircase \cite{HuWe} there
are uncountably many invariant measures, indexed by a positive real
`deformation' parameter. Hooper \cite{Ho-classification} exhibits some
uniquely ergodic cases but these are not typical. 

\item K. Fraczek and C. Ulcigrai \cite{Franczek-Ulcigrai} provide examples
  of recurrent $\Z$-covers $\til M \to M$ which do admit a strip,
  where $\til M$ is {\em not} a lattice surface, and for which the
  straightline flow is {\em not} ergodic in almost every direction. In
their examples $M$ is a square-tiled surface of genus two, and the
Veech group of $\til M$ is infinitely generated of the first kind. 
It
follows from Theorem \ref{thm: main} that in these examples, the set of
ergodic directions is a dense $G_\delta$, see Corollary \ref{cor: FrUl}.  

\end{enumerate}
\end{remark}
\subsection{Applications}
It is of interest to construct dense trajectories for vector fields on
noncompact surfaces. For example Panov \cite{Pa} constructed a dense
trajectory for 
a vector field on the plane, which is a pullback under a covering map,
of a constant slope field on a torus endowed with a quadratic 
differential. Since ergodic flows with respect to a measure of full
support have dense trajectories, our theorem gives many new examples
of infinite translation surfaces for which the straightline flow has dense
trajectories.

The question of ergodicity of a $\Z$-valued skew product over $f: X
\to X$ is
only well-understood for a limited class of $X$. For example the case
in which $X = S^1$ is the circle and $f$ is an irrational rotation is
well-understood (see \cite{ANSS} and the references therein). This is
equivalent to the case in which $X$ is an interval and $f$ is an
interval exchange on 2 intervals. On the other hand the question for interval
exchange transformations on an arbitrary number of intervals is
challenging and poorly
understood. There are not many examples of $\Z$-valued
skew products over interval exchanges which are known to be
ergodic. Our Theorem \ref{cor: Veech dichotomy} is a source of many
new examples. Suppose $\til M \to M$ is a recurrent $\Z$-cover which is
a lattice surface containing an infinite strip, and $\theta$ is an
ergodic direction on $M$. Let $I$ be a segment in $M$ transverse to
direction $\theta$ and $\til I$ a segment which is a preimage of $I$
in $\til M$. If $f: I \to I$ (resp. $\til f: \til I \to \til I$) is
the return map to $I$ (resp. $\til I$) along lines in direction
$\theta$ in $M$ (resp. $\til M$) then $\til f$ is well-defined 
in light of Proposition \ref{prop: Zcovers correspondence}(ii), 
and is a $\Z$-valued skew product over
$f$. If $\til M \to M$ satisfies the
conditions of Theorem \ref{thm: main} then $\til f$ is ergodic. 
An alternative approach to the ergodicity question for
$\Z$-valued skew products over {\em periodic} interval exchanges
(i.e. those for which the geodesic flow is periodic), under an 
assumption on the gap in the Lyapunov spectrum, is developed in
the recent work \cite{Conze-new}.

A well-known result of Kesten \cite{Kesten} implies that on the
infinite staircase surface, there are no bounded 
trajectories (i.e. trajectories for the straightline flow which in all
positive times are confined between two levels of the staircase). By
an argument of \cite{Peterson} which we recall 
in Proposition \ref{prop: Kesten}, Theorem \ref{thm: main}
provides more examples of recurrent $\Z$-covers without bounded
trajectories. Note that Ralston \cite{Ralston} has constructed
trajectories on the infinite staircase which are bounded above. \\



\section{background}

In this section, we briefly introduce our notation and state the results we will need.


\subsection{Translation surfaces}

A surface is called a \emph{translation surface} if it can be obtained
by edge-to-edge gluing of polygons in the plane, only using translations
(the polygons need not be compact or finitely many but should be
at most countably many, and the lengths of sides and diagonals should
be bounded away from zero). Two translation surfaces are considered
equivalent if the corresponding decompositions into polygons have a
common locally finite refinement. 
The translation structure induces a flat
metric with conical singularities; in the non-compact case, `infinite
angle singularities' may arise. The translation structure also induces
horizontal and vertical transverse measures $dx, dy$, and a Lebesgue
measure $dxdy$.

Let $M$ be a compact translation surface. Sometimes we add finitely
many marked
points to the set of singularities, i.e. distinguished points at which
the cone angle is $2\pi$ which we consider as singularities. To avoid
uninteresting complications we always assume that the set of
singularities is nonempty.    
A \emph{saddle connection} is a geodesic segment for the flat metric
starting and ending at a singularity, and containing no singularity in
its interior.
A \emph{cylinder} on a translation surface is a maximal connected union
of homotopic simple closed geodesics. We call the length of such a
closed geodesic the {\em circumference} of the cylinder, and the length of a
perpendicular segment going across the cylinder, its {\em height}. 
Let $M$ be a compact translation surface, let $P \subset M$ be a
finite set of points containing the singularities, and let $\delta$ be
a curve on $M$ which is either closed or connects points of $P$. 
We denote by 
$\hol(\delta)$ the holonomy of the translation structure on $M$ along
$\delta$; i.e. 
the vector in $\R^2$ obtained by integrating $dx$ and $dy$ elements
along $\delta$. 
The holonomy map is well defined on $H_{1}(M, P; \Z)$. 

Given any translation surface $M$, an \emph{affine
diffeomorphism} is an orientation preserving homeomorphism of $M$ that
permutes the singularities of the flat metric and acts affinely on the
polygons defining $M$.  The group of affine diffeomorphisms is denoted
by $\Aff(M)$.  The image of the derivation map
$$d : 
   \left\{
   \begin{matrix}
    \Aff(M) \to \GL(2,\R) \\
    f \mapsto df
    \end{matrix}
   \right.
$$
is called the \emph{Veech group}, and denoted by $\Gamma(M)$.  If
$M$ is a compact translation surface, then $\Gamma(M)$
is a discrete subgroup of $G = \SL(2,\R)$.  

A translation surface is a \emph{lattice surface} if its Veech group
is a lattice in $G$. Note that lattice surfaces are sometimes referred
to as \emph{Veech surfaces}. 
%
%
%
%
We say that two subgroups
$\Gamma_1, \Gamma_2$ of a group $G$ are {\em commensurable} if they
share a common finite index subgroup. Let 
$\overline{\phi}_{t}^{\theta}$ denote the straightline flow in 
direction $\theta$ on $\til M_{w}$ and by
$\phi_{t}^{\theta}$ the flow on $M$.  
A straightline flow on a translation surface is {\em
uniquely ergodic}  if the only invariant probability measure is
Lebesgue measure.   
If the linear flow $\phi_{t}^{\theta}$ is uniquely ergodic, we say that
the direction $\theta$ is uniquely ergodic.

Let $G = \SL(2,\R)$, and let 
\eq{defn flows}{
 g_t = \left(\begin{array}{cc}
e^{t} & 0 \\ 0 & e^{-t} \end{array} \right), \ \ \ 
r_{\theta} =
\left(\begin{array}{cc}
\cos \theta & -\sin \theta \\
\sin \theta & \cos \theta
\end{array}
\right).  
}
There is a moduli space or {\em stratum} of translation surfaces of a fixed genus and
singularity pattern, and this space is equipped with a $G$-action. 
The flow induced by the action of $\{g_t\}$ is called the geodesic
flow. For any flow on a topological space, 
we say that a trajectory $\{g_{t}x\}$ is divergent 
if for any compact subset $K$ of the space there is $t_0$ such that
$g_t x \notin K$ for all $t>t_0$. 
Masur's well-known criterion provides a link between the dynamics of
the straightline flow on a surface and the dynamics of the geodesic
flow on its stratum of translation surfaces. Throughout this paper we set $\theta' = \pi/2-\theta.$

\begin{prop}[{Masur condition \cite{Masur-Duke}}]\name{prop: Masur
condition}
If $M$ is a compact translation surface and $\{g_t r_{\theta'} M \}$
is not divergent in the moduli space of
translation surfaces, then $\theta$ is a uniquely ergodic direction on
$M$. In particular, denoting by $\Gamma$ the Veech group of $M$ and by
$x \mapsto [x]$ the natural map $G \to G/\Gamma$, if $\{g_t [r_{\theta'}] \}$
is not divergent in $G/\Gamma$, then $\theta$ is a uniquely ergodic direction on
$M$.

\end{prop}

%
%
%


\subsection{Fuchsian groups, geodesic flow, and approximation by cusps}\name{subsection: cusps}
Let $\Gamma$ be a Fuchsian group, i.e. a discrete subgroup of $G$. In this
subsection we recall some facts we will need --- see \cite{Ka} for an
introduction to Fuchsian groups. 

 We denote the upper half-plane by $\HH$, so
that $G$ acts on $\HH$ by M\"obius transformations. For us it
will be convenient to use the {\em right-action}
\eq{eq: convenient}{
z.g = \frac{dz-b}{-cz+a}, \ \ \mathrm{where} \ g =
\left(\begin{matrix} a & b \\ c & d \end{matrix} \right) \in G.
} 
A horoball in $\HH$ is the isometric image of 
$$\HH^+_c=\{z \in \HH: \Im z >c\}$$
for some $c>0$, where $\Im z$ denotes the imaginary part of $z$. A
parabolic element of $G$ is any matrix of trace 2 
other than the identity. If
nontrivial, the stabilizer of a horoball in $\Gamma$ is a maximal
unipotent subgroup, i.e. an infinite cyclic group generated by a
parabolic element which is maximal with respect to inclusion. 
A cusp in $\HH/\Gamma$ is the image of the map 
\eq{eq: horoball}{
B/P \to \HH/\Gamma,}
where $B$ is a horoball, $P$ is the stabilizer
of $B$ in $\Gamma$ and is nontrivial, and $B$ is maximal with respect to the
property that the map \equ{eq: horoball} is
injective. Any maximal parabolic subgroup $P \subset \Gamma$
stabilizes a cusp, which is a proper subset of $G/\Gamma$ when
$\Gamma$ is non-elementary. 
$\Gamma$ is called a non-uniform lattice if $G/\Gamma$ is not compact
but has finite measure. In this case $\HH/\Gamma$ has a finite nonzero
number
of cusps, and their complement is a compact subset of
$\HH/\Gamma$. The map sending a cusp in $\HH/\Gamma$ to the conjugacy 
class of the group $P$ as in \equ{eq: horoball} is well-defined and
induces a bijection between cusps and conjugacy classes of maximal
parabolic subgroups. 

 We have a map $G/\Gamma \to
\HH/\Gamma$ which maps $[g]$ to the image in $\HH/\Gamma$ of
$\mathbf{i}.g$, where $\mathbf{i} = \sqrt{-1}.$ In case $\Gamma$ is
torsion-free, the quotient $\HH/\Gamma$ is a noncompact finite-volume
hyperbolic manifold, $G/\Gamma$ can be identified with its unit
tangent bundle, and the above projection is the projection mapping a
unit tangent vector to its basepoint. 
The group $G$ acts on $G/\Gamma$ via left translations $g_1[g_2] =
[g_1g_2]$, and the action of the group $\{g_t\}$ gives rise to the
geodesic flow on the unit tangent bundle. Let $\cc$
be a cusp in $\HH/\Gamma$, which is the image of a horoball under a
map \equ{eq: horoball}, where $B$ is the isometric image of $\HH_{c_0}^+$
for some $c_0>0$. For any $c>c_0$ let $\cc_c$ be
the image of $\HH_{c}^+$ under the same maps. That is, the sets $\{ \cc
\sm \cc_c: c>c_0\}$, give an exhaustion of $\cc$ by
bounded sets. We say that a geodesic
orbit $\{g_tx\}$ {\em penetrates infinitely often into $\cc_c$} if
there is $t_n \to \infty$ such that $g_{t_n}x \in \cc_c$
for all $n$.

Another action of $G$ is its linear action on the punctured plane
$\R^2 \sm \{0\}$. Under
this action, a point $x$ has a discrete orbit under a non-uniform
lattice $\Gamma$ if and only if its
stabilizer contains a parabolic element. For any $x$, the set
of directions of vectors in the orbit $\Gamma x$ is dense in $S^1$. To
quantify such a density statement, 
for a discrete orbit $\Gamma x$,
$\theta \in S^1$ and $d>0$, we say that $\theta$ is {\em
$d$-well-approximable by $\Gamma x$} if there are infinitely
many $\gamma \in \Gamma$ for which  
\eq{eq: approximates}{
\|\gamma x\| \, |\E_\theta \wedge \gamma x| < d.
}
Here $\|\cdot \|$ denotes the Euclidean norm on $\R^2$ and $|u \wedge
v| = \|u\|\|v\||\sin \theta|$ where $\theta$ is the angle between the
directions of $u$ and $v$. 

\begin{prop}\name{prop: geodesic}
Let $\Gamma$ be a non-uniform lattice, let $x \in \R^2 \sm \{0\}$ such
that $\Gamma x$ is discrete, and let $\cc$ denote the cusp in
$G/\Gamma$ corresponding to the stabilizer of $x$.   
For any $c$ there is $d$ such that if $\theta \in S^1$ is
$d$-well-approximable by $\Gamma x$ then $\{g_t[r_{\theta'}]\}$ penetrates infinitely
often into $\cc_c$. Conversely, for any $d$ there is $c$ such that if 
 $\{g_t[r_{\theta'}]\}$ penetrates infinitely
often into $\cc_c$ then  $\theta$ is
$d$-well-approximable by $\Gamma x$.

\end{prop}

\begin{proof}
This is a standard computation (see e.g. \cite{Patterson}). To keep
the paper self-contained, we prove the implication $\implies$, leaving
the converse to the reader. 

Let $g_0 \in G$ such that $g_0x = \E$, where $\E = (1,0)$ is the first
standard basis vector. Replacing $\Gamma$ with $g_0\Gamma g_0^{-1}$, we
may assume that $x = \E$, so that the cusp corresponding to $x$ is
$\pi(\HH^+_{c_0})$ for some $c_0$. Given $c>0$, we may assume without
loss of generality that $c>c_0$, and we let $d = 1/2c.$ 
Suppose that for some sequence $\gamma_n \in
\Gamma$, \equ{eq: approximates} holds for all $\gamma=\gamma_n$. 
Let $T_n, \theta_n$ denote respectively the length and angle of $\gamma_nx$,
i.e. $\gamma_nx= T_n\E_{\theta_n}$, and write $\displaystyle{
\gamma_n = \left(\begin{matrix}
a_n & b_n \\ c_n & d_n
\end{matrix}
\right), 
}$
so that
\[
\gamma_n x = \gamma_n \E =
\left(\begin{matrix} a_n \\ c_n
\end{matrix}
 \right) = T_n \E_{\theta_n}
.
\]
Then \equ{eq:
  approximates} is equivalent to
\eq{eq: we get}{
T_n^2\, |\sin(\theta-\theta_n)|<d.
}
It follows easily from \equ{eq: convenient} that for any $g \in G$,
$\Im(\mathbf{i}.g) = \frac{1}{a^2+c^2}$. 
Let $t_n = \log T_n - \log \sqrt{d}$. If we write 
$$
\left(\begin{matrix} 
\bar{a}_n & \bar{b}_n \\ \bar{c}_n & \bar{d}_n
\end{matrix}\right) = g_{t_n}r_{\theta'} \gamma_n,
\mathrm{ with } \  \theta' = \pi/2
 $$
 $$
 \mathrm{then}
\ \left( \begin{matrix}\bar{a}_n \\ \bar{c}_n \end{matrix}\right) = 
g_{t_n} r_{\theta'} T_n \E_{\theta_n} = 
\left(\begin{matrix}
T_n^2\sin(\theta-\theta_n)/\sqrt{d} \\ \sqrt{d} \cos(\theta - \theta_n)
\end{matrix} \right).
$$
This implies via \equ{eq: we get} that
$$
\Im(\mathbf{i}. (g_{t_n}r_{\theta'}\gamma_n)) =
\frac{1}{
\bar{a}_n^2+\bar{c}_n^2 }
> \frac{1}{2d} =c
. 
$$ 
This shows that for all $n$, $g_{t_n}r_{\theta'} \gamma_n \in \HH^+_c$, as 
required.  
\end{proof}

\begin{prop}\name{prop: a.e. and hd}
For any non-uniform lattice $\Gamma$, any $d>0$, and any $x \in \R^2
\sm \{0\}$ for which $\Gamma x$ is discrete, the set 
$$\Theta_d=\left\{\theta \in S^1: \theta \mathrm{\ is \ not \
}d\mathrm{-well-approximable \ by \ } \Gamma x\right \}$$
has zero Lebesgue measure, and moreover its Hausdorff dimension is less
than 1. 
\end{prop}

\begin{proof}
We first prove that $\Theta_d$ has zero Lebesgue measure. The geodesic
flow on $G/\Gamma$ is ergodic, and for any $c>0$, the cusp $\cc_c$ has
positive measure, with respect to the $G$-invariant probability
measure $\mu$ on $G/\Gamma$ induced by Haar measure on $G$. This implies that for any $c>0$, 
$$\Omega_c = \{x \in G/\Gamma : \{g_tx\} \mathrm{\ does \ not
  \ penetrate \ into \ }\cc_c \mathrm{\ infinitely \ often} \}  
$$
has measure zero. 
On the other hand, by Proposition \ref{prop: geodesic}, for some $c>0$, and some cusp
$\cc$ in $G/\Gamma$, $\Omega_c$ 
contains $\{[r_{\theta'}]: \theta\in \Theta_d\}.$ Let 
$\displaystyle{h_s^-  = \left(\begin{matrix} 1 & 0 \\ s & 1 \end{matrix} \right)}.$ 
Then $g_t h_s^- g_{-t} \to_{t \to \infty} \mathrm{Id}$ in $G$ and this
implies that if $x \in \Omega_{c}$ then $h^-_sx \in
\Omega_{c'}$ for any $s$ and any $c'>c$.  Haar measure is smooth on $G$ and the map 
$(s, t, \theta) \mapsto g_t h^-_s r_{\theta}$ is a local
homeomorphism. This implies that for any $c'>c$, the set $\Omega_{c'}$
contains
$$
\{\pi(g_t h_s^- r_{\theta'}) :  \theta \in \Theta_d\}, 
$$
which has $\mu$-positive measure whenever $\Theta_d$ has
positive Lebesgue measure. This proves the claim. 

The stronger result
about Hausdorff dimension would follows by a similar argument, provided we
knew that 
for any $c>0$, the Hausdorff dimension of $\Omega_c$ is less than
3, which is the dimension of $G/\Gamma$. In case $G/\Gamma$ has one cusp, 
this follows from \cite[Prop. 8.5]{margulis}. As explained to the authors by
Manfred Einsiedler \cite{Einsiedler}, 
the general case of the latter statement can be deduced from
his recent work with Kadyrov and Pohl \cite{EKP}.
\end{proof}

For a geodesic trajectory $\{g_tx\}$, either its projection in
$\HH/\Gamma$ belongs to a cusp for all $t>t_0$, or it returns
infinitely often to the complement of the cusps. This implies:

\begin{prop}\name{prop: excursion}
If $\Gamma$ is a non-uniform lattice, for all but countably many
directions $\theta$, the trajectory $\{g_t [r_{\theta'}]\}$ does not
diverge in $G/\Gamma$. 
\end{prop}

\subsection{$\Z$-covers} \label{section:Zcover}
We recall some terminology
 from  \cite{Zcovers}.  Let $M$ be a compact translation surface.
 We fix a finite set $P$ of points containing the singularities. Denote the
relative homology by  $H_{1}(M, P; \Z)$  and  
 $H_{1}(M \sm P, \Z)$ the absolute homology of $M$ puctured at
$P$. The intersection form is a nondegenerate bilinear form   
 $$i: H_{1}(M, P; \Z) \times H_{1}(M \sm P, \Z) \to \Z.$$
The $\Z$-cover $\til M_{w}$ of $M \sm P$ associated to a
non-zero $w$ in $H_{1}(M, P, \Z)$ 
is the cover associated to the  kernel of the homomorphism 
$$\phi_{w}: \pi_{1}(M\sm P) \to \Z, \ \ \ \delta \mapsto i(w,
[[\delta]]),$$ 
where $[[\delta]]$ denotes the homology class of $\delta$. 

For a translation surface $M$ and $w \in H_1(M, P; \Z)$, let
$\mathrm{hol}(w) \in \R^2$ denote the holonomy along $w$; that is the
vector whose coordinates are respectively the integrals of the
elements $dx, dy$ along any path representing $w$. 
We have:

\begin{prop}\name{prop: Zcovers
correspondence} 
\begin{itemize}
\item[(i)]
There is a bijective correspondence between $\Z$-covers $\til M \to M$
and projective classes of cycles $w \in H_1(M, P; \Z)$. 
\item[(ii)]
The holonomy $\mathrm{hol} (w)$ vanishes if and only
if for any
$\theta$ for which $\phi_{t}^{\theta}$ is ergodic,
$\overline{\phi}_{t}^{\theta}$ is recurrent. 
\item[(iii)]
The Veech group of $\til M_w$ is Fuchsian. If it is non-elementary then
$\mathrm{hol}(w)=0$.  
\item[(iv)]
$\Gamma_{\til M}$
contains a finite index subgroup which descends to $\Gamma_M$. 
\item[(v)]
If $\{g_t [r_{\theta'}] \}$ is not divergent in $G/\Gamma_{\til M}$,
and $\til M \to M' \to M$ is an intermediate finite cover of $M$, then
$\{g_t [r_{\theta'}] \}$ is not
divergent in $G/\Gamma_{M'}.$ 
\end{itemize}
\end{prop}
In (ii), by recurrence we mean that for any measurable $A \subset \til M$, for
a.e. $x \in A$ there is $t_n \to \infty$ such that $\overline{\phi}_{t_n}x \in
A$. If (ii) holds we say that $\til M$ is a {\em recurrent
$\Z$-cover.} 

\begin{proof}
Items (i)--(iv) are proved in \cite{Zcovers}, and (v) follows from 
(iv) and the fact that a subgroup of $\Gamma(\til M)$
which descends to $\Gamma(M)$ also descends to $\Gamma(M')$ for any
intermediate finite cover $M'$. 
\end{proof}

Propositions \ref{prop: Masur condition}, \ref{prop: excursion} and
\ref{prop: Zcovers 
correspondence}(v) imply one
of the hypotheses of Theorem \ref{thm: main}:

\begin{cor}\name{cor: intermediate}
Suppose $\til M \to M$ is a $\Z$-cover and $\theta$ is a direction
such that $\{g_t[r_{\theta'}]\}$ is not divergent in $G/\Gamma$, where
$\Gamma$ is the Veech group of $\til M$. Then $\theta$ is ergodic on
intermediate covers. In particular, if $\til M$ is a lattice surface,
then $\theta$ is ergodic on
intermediate covers for all but countably many $\theta$. 
\end{cor}
\qed

We give a more concrete description of the construction of the cover
$\til M_w$. One can
represent $w$ as $ \sum a_j \delta_j$, where $a_j \in \Z$, and the
$\delta_j$ are finitely many 
disjoint oriented arcs on $M$ which are either closed or have endpoints in
$P$. Now form countably many copies $\{M_k: k \in \Z\}$ (each with its
copies of the arcs $\delta_j$), and glue each
$M_k$ to $M_{\ell}$ along the left (resp. right) of $\delta_j$ if $\ell =
k+ a_j$ (resp. $\ell = k-a_j$). 
We now give a necessary and sufficient condition for the existence
of infinite strips on $\til M$. 
\begin{prop}\name{prop: having strips}
Let $\til M_{w} \to M$ be a recurrent $\Z$-cover of translation
surfaces associated to the homology class $w$. 
$\til M$ has an infinite strip if and only if there is a
cylinder $C$ on $M$, such that the homology class $[[\delta]]$ represented
by the core curve $\delta$ of $C$ satisfies $i(w, [[\delta]]) = k \neq
0.$ In this case $S^k$ maps the strip to itself, and is a translation along the
direction of the strip. 
\end{prop}

\begin{proof}
Given a cylinder
$C$ on $M$, a connected component of its lift to 
$\til M$ is a cylinder or a strip. Let $\delta$ be the core curve of
$C$. The cylinder $C$ lifts to a cylinder on $\til M$ if and only if
preimages of 
$\delta$ are compact loops. By definition of the covering, this holds
if and only if $i(w, [[\delta]]) = 0.$  
Conversely, if $\Sigma \subset \til M$ is an infinite strip, let $\ell$ be an
infinite straight line in its interior. Note that the projection of
$\Sigma$ to $M$ does not contain singularities, so that the image of
$\ell$ in $M$ is a straight line whose distance to singularities
remains bounded below. A basic fact on compact translation surfaces
(see e.g. \cite[\S1.6]{MaTa}) asserts that the projection $\delta$ of $\ell$
must be periodic, i.e. is the core curve of a cylinder $C$. That is the
image of $\Sigma$ covers $C$, and by the above, $i(w, [[\delta]]) \neq
0.$

Assume now that $i(w, [[\delta]]) = k \neq 0$. This exactly means that a
lift $\til \delta$ of $\delta$ starting at level 0 on $\til M$ ends
at level $k$. Thus each connected component of   
$\displaystyle\bigcup_{n \in \Z}S^n(\til \delta)$
is an infinite line in $\til M$ passing through an infinite strip, and
the action of $S^k$ is by translation along the line. 
\end{proof}





\subsection{Cocycles and essential values}
Suppose $F_t: X \to X$ is a probability measure preserving flow on a
non-atomic   
measure space $(X, \mu)$, and suppose $\alpha: X \times \R \to
\Z$ is a measurable cocycle, i.e., a function satisfying 
 \eq{eq: cocycle identity}{
\alpha(x, t+s) = \alpha(x,t)+\alpha(F_tx, s).
}
Form the space $X_\alpha = X 
\times \Z$ with the obvious measure, and the flow $\ov{F}_t:
X_{\alpha} \to X_{\alpha}$ 
defined by 
$$\ov{F}_t(x,n) = \left(F_tx, n+\alpha(x,t) \right).
$$
This flow is called a {\em $\Z$-valued skew product} (over $X$,
corresponding to 
$\alpha$). If two cocycles $\alpha$ and $\beta$ are cohomologous,
i.e. if there is a measurable $g: X \to \Z$ such that $\alpha(x,t)
= \beta(x,t)+g(F_tx) - g(x),$ then the skew products are
measurably equivalent flows.  
A number $k \in \Z$ is called an {\em essential value} for this skew
product if for any measurable $A \subset X$ with
positive measure, there is a set of $t$ of positive measure for which 
$$\mu \left \{x \in A : F_{t} x \in A, \, \alpha(x,t)=k \right\}>0.$$
There is a $\Z$-action on $X_\alpha$ obtained from the action of $\Z$ on
itself by addition in the
second factor; clearly it commutes with the $\ov{F}_t$-action, so for
every subgroup $k\Z \subset \Z$ we can form the space $X_\alpha /
k\Z$, and we have 
$$X_\alpha \to X_\alpha / k\Z \to X_\alpha / \Z
\simeq X.$$

\begin{prop}[{Schmidt \cite[Chap. 3]{Schmidt}}]\name{prop: essential values}
Suppose $k$ is an essential value. Then $\ov{F}_t$ is ergodic on
$X_\alpha$ if and only if it is ergodic on $X_{\alpha} / k\Z$. 
\end{prop}

\section{From strips to essential values}\name{section: essential}
In this section we prove Theorem \ref{thm: main}. The main point will
be to approximate the flow in direction $\theta$ with the flow in a
strip, where it does not encounter discontinuity points. The condition
that the direction $\theta$ is well-approximated by strips ensures
that the flow in a strip remains close to the flow in direction
$\theta$ for a sufficiently long time, enabling us to produce
essential values. 

We will need some additional
notation. Let $M$ be a compact translation surface and
$\til M \to M$ a recurrent $\Z$-cover. For a direction $\theta$, $\phi_t$ and
$\ov{\phi}_t$ denote the straightline flows in direction $\theta$ on
$M$ and $\til M$ respectively (note our notation does not reflect the
dependence
on $\theta$). It is easy to see that $\ov{\phi}_t : \til M \to \til
M$ is a skew product over $\phi_t : M \to M$. Namely,  fix some
$x_0$ in $M$ and for each $x \in M$, a continuous path $\beta_{x_0,
x}$ from $x_0$ to $x$, and let $\beta_{x, x_0}$ denote the path with
the same trace in the opposite direction. Let $w \in
H_1(M, P; \Z)$ such that $\til M = \til M_w$ as in Proposition
\ref{prop: Zcovers correspondence}, and let $i$ be the intersection
form on $M$. Now define a cocycle  
\eq{eq: cocycle}{
\alpha (x,t) = i(w, \delta),
}
where $\delta$ is the path from $x_0$ to $x$ along $\beta_{x_0,x}$,
followed by moving from $x$ to 
$\phi_tx$ along the line in direction $\theta$, followed by
$\beta_{\phi_tx,x_0}$. It is not hard to show that $\alpha$ is a cocycle
and that making a different choice for $x_0$ and the $\beta_{x_0,x}$ would result
in a cohomologous cocycle. Moreover it follows from the description of
$\til M_w$ that the straightline flow $\ov{\phi}_t : \til M\to \til
M$ is measurably equivalent to the lift of $\phi_t$ to this skew
product (see \cite[Proof of Proposition
15]{Zcovers}). 

Since $\theta$ is well-approximable by cylinders, there are $k \neq
0, \,\vre>0$, and for $n=1,2, \ldots$ there are strips 
$\Sigma_n$ in $\til M$, such that the corresponding $v_n =
v(\Sigma_n)$ and $A_n= A(\Sigma_n)$ satisfy $k \equiv k(\Sigma_n)$ and 
\eq{eq: approximation1}{
|\E_\theta \wedge v_n| \leq (1-\vre) \, \frac{A_n}{2\|v_n\|} \ \ \ \ \ \mathrm{and}
\ \ \ A_n \geq \vre.
}
In light of Proposition \ref{prop: essential values}, Theorem
\ref{thm: main} follows immediately from: 

\begin{claim}
\name{claim: essential value}
$k$ is an essential value for the flow $\{\bar{\phi}_t\}$.
\end{claim}

The rest of this section is devoted to proving Claim \ref{claim:
essential value}. Let $R$ be a rectangle, and for $c>0$, let $cR$ be a
concentric rectangle which is obtained from $R$ by dilating it by a
factor of $c$. By saying that $\til M$ contains a rectangle $R$ we
mean that there is an isometry which is a translation 
in charts, mapping $R$ to $\til M$; in particular the image
of $R$ under this isometry does not contain a singularity. When we
refer to the corners of $R$, we mean the images of the corners under
the above isometry. 
 
Let $p: \til M \to M$ be the covering map, and let
$C_n=p(\Sigma_n)$. Note that $p^{-1}(C_n)$ is the union of images of
$\Sigma_n$ under the deck group. Let $\theta^* = \theta + \pi/2$ be
the direction perpendicular to $\theta$. 
\begin{Def}\name{def: admits}
 We will say that $x \in \til M$ {\em
admits a rectangle at stage $n$} if $p^{-1}(C_n)$ contains $cR$,
where $R$ is the closed rectangle with sides in directions $\theta,
\theta^*$ and opposite corners at $x, \, S^{k}x,$ and $c =
(1-\vre/2)(1-\vre)^{-1}>1$.  
\end{Def}
We denote by $\til \ca$ the set of $x$ which admit a rectangle at
stage $n$, for infinitely many $n$ and by $\ca$, its projection to $M$.

\begin{lem}\name{lem: conull invariant}
The set $\ca$ is conull (i.e. its complement has measure zero).
\end{lem}

\begin{proof}
 On the compact surface $M$, the number
of cylinders with a given bound on the length of their core curve is
bounded. Therefore $\|v_n\| \to \infty$, so by
\equ{eq: approximation1}, the
direction of $v_n$ tends to $\theta$. 

Let 
$\ell_n$ be
a core curve of $C_n$ at equal
distances from its two boundaries, let $\bar{\ell}_n =
p^{-1}(\ell_n)$, and let $\Sigma'_n$ be the points of 
$p^{-1}(C_n)$ which are within distance  
$\eta_n=\displaystyle{\frac{\vre^2}{8\|v_n\|}}$ of $\bar{\ell}_n$. Note 
that $v_n$ is the diagonal of a rectangle in $\Sigma_n$ with corners
at $x, S^kx$ and $\|v_n\|$ is 
the circumference of $C_n$. This implies that  
\eq{eq: implies}{
h_n=\frac{A_n}{2\|v_n\|}\geq \frac{\vre}{2\|v_n\|}
}
is the distance from
$\ell_n$ to the boundary edges of $C_n$, hence from $\bar{\ell}_n$ to
the boundary edges of $p^{-1}(C_n)$. Therefore $c|\E_\theta \wedge
v_n|$ is the length of the side of a rectangle $cR$ as in Definition
\ref{def: admits}, in direction $\theta^*$.   
Inequality \equ{eq: approximation1} says that this
side has length at most $(1-\vre/2)h_n$, so that any point which
is of distance 
$$2\eta_n = \frac{\vre^2}{4\|v_n \|} \stackrel{\equ{eq:
implies}}{\leq} \frac{\vre h_n}{2}$$
 from $\bar{\ell}_n$ admits a rectangle at stage
$n$.

In particular any point of $\Sigma'_n$ admits a rectangle at stage
$n$. Let $C'_n = p(\Sigma'_n)$. By construction, each $\Sigma'_n$ is
deck group invariant, i.e. $\Sigma'_n = p^{-1}(C'_n)$.  
Now we let $\Sigma'=\limsup \Sigma'_n$ denote the set of points in
$\til M$ which belong
to infinitely many of the $\Sigma'_n$, and $C' = \limsup
C'_n$. Clearly $\Sigma' \subset \til
\ca$, and by construction, $\Sigma' = p^{-1}(C')$. For each $n$, the measure of each
$C'_n$ in $M$ is at least $\vre^2/4$, so $C'$
has positive measure and is contained in $\ca$. To prove the lemma, by ergodicity of
$\{\phi_t\}$ on $M$,
it suffices to show that $\bigcup_{t} \phi_t(C'),$ which is a
$\{\phi_t\}$-invariant set, is contained in $\ca$. So we fix $x \in
C'$ and $t \in \R$, and consider $x'=\phi_tx$. We have shown that there are infinitely
many $n$ for which $x$ is within $\eta_n$ of $\ell_n$, which is a
curve parallel to $v_n$. Since the direction of $v_n$ tends to
$\theta$, for all large enough $n$, $x'$ is within $2\eta_n$ of
$\ell_n$, so admits a rectangle at stage $n$. That is, $x' \in \ca$.
\end{proof}

Continuing with the proof of Claim \ref{claim: essential value}, let
$A \subset M$ be a measurable set of positive measure, and let $x_{0}$
be a Lebesgue density point of $A \cap \ca$. Let $x$ be any of the
preimages of $x_0$ in $\til M$. There is a sequence of
rectangles $R_n$ in $\til M$ which have $x,\, S^kx$ at opposite
corners, have sides in directions $\theta, \theta^*$, and such that
$cR_n$ is embedded in $\til M$. Note that the
side of $R_n$ in direction $\theta$ is getting longer with $n$ and the
side in direction $\theta^*$ is getting shorter; we will refer to
these as the long and short side of $R_n$ respectively, and denote
their lengths by $a_n, b_n$. 
We choose a fundamental domain measurably 
isomorphic to $M$ such that $x$ belongs to level 0 and $S^k(x)$
belongs to level $k$ (here we may use the concrete description of the
covering given in \S\ref{section:Zcover}), and we identify $A$ with a
subset of this fundamental domain.  

Let 
$Q_0$ be the square centered at $x$ with sides in directions
$\theta$ and $\theta^*$, such that the sidelength of $Q_0$ is
equal to $b_n$. 
Denote 
$$Q_1= S^k(Q_{0}) \ \ \mathrm{and} \ \ Q_2(t) = \bar{\phi}_t(Q_0),$$
where $t$ is a parameter 
(note that $Q_0, Q_1, Q_2$ all depend on $n$ but we omit this to
simplify notation). 

Following the notations in figure \ref{fig:rectangles}, for $n$ large enough, the right half of  $Q_0$ denoted by $P_0$ is  contained $cR_n$ and thus embedded in $\til M$. 
Moreover, by construction $cR_n$ does not contain any
singular point in its closure. This implies that if $t <t_n= ca_n -
(c-1)b_n$, then the restriction of 
$\bar{\phi}_t$ to $P_0$ is continuous, i.e. the right half of $Q_2(t)$ denoted by $P_2(t)$ is also a 
rectangle embedded in $\til M$. By the same reasoning the left hand side of $Q_1$ is a rectangle embedded in $\til M$.

Let $\lambda$ denote Lebesgue measure on $\til M$. We have
$\lambda(P_0) = \lambda(P_1) = \lambda(P_2).$ Moreover when $t=a_n$ we
have $P_1$ and $P_2=P_2(t)$, and obtain 
$\lambda\left(P_1 \cap P_{2}\right)
= \lambda(P_{0})$. See figure \ref{fig:rectangles}.

\begin{figure}[h!]
\begin{center}

\begin{tikzpicture}

 \begin{scope}[shift={(3cm,-5cm)}, fill opacity=0.5]
  \draw[black,very thick] \rectangleone;
 
 \draw[black,very thick] \rectangletwo; 
  \draw[black,very thick] \rectanglethree; 
   \draw[black,very thick] \rectanglefour; 
    \draw[black,very thick] \rectanglefive;

  \fill[red]\rectangletwo;
  
    \fill[blue] \rectanglethree;
   \fill[yellow] \rectanglefour;
    \fill[green] \rectanglefive;

\end{scope}

\end{tikzpicture}
\caption{\label{fig:rectangles}
Rectangles. The direction $\theta$ is vertical. The red rectangle
is $R_n$, its bottom left corner is $x$, its upper right
corner is $S^k(x)$. The blue rectangle is $Q_0$, the upper right
rectangle is $Q_1$, and the upper left
rectangle is $Q_2$.
}
\end{center}
\end{figure}
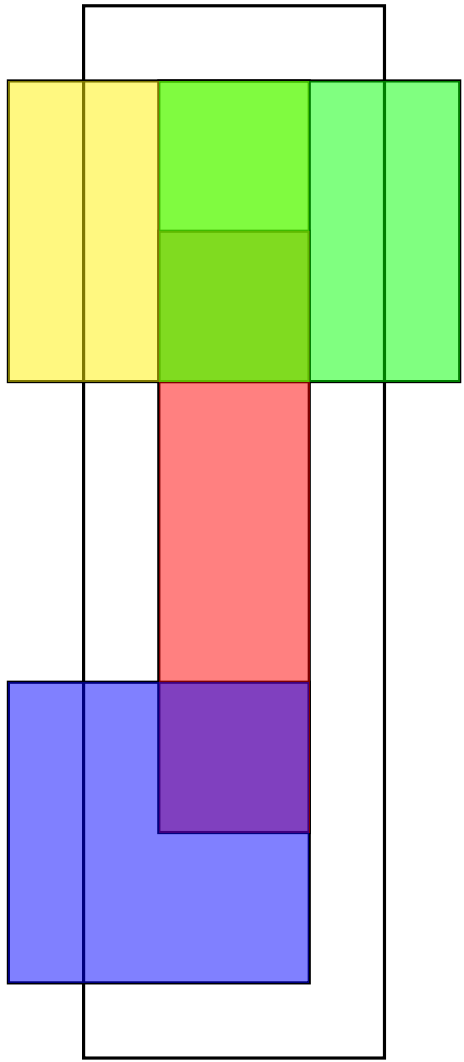

As $x_{0}$ is a density point of $A$, for every square $Q$ centered at
$x_{0}$ of sufficiently small diameter,  
$\frac{\lambda(Q \cap A)}{\lambda(Q)} \geq 15/16$. This inequality is
fulfilled for $Q_0$ for large enough $n$, since  
$b_n \to 0.$
A straightforward computation shows that the same is true for $P_0$.

Note that $t_n>a_n$ for all large $n$. 
Therefore, there is an interval $I$ centered at $a_n$ such that
if $t \in I$, 
$$
\frac12< b \df \frac{\lambda\left(P_1 \cap P_2(t)
\right)}{\lambda\left(P_{0}\right)} \leq 1.$$

In order to show that $k$ is an essential value, it will suffice to
show that $A_k \cap A^{(t)}$ has positive measure for all $t \in I$,
where 
$$A_k = A \times \{k\} \ \ \mathrm{and} \  A^{(t)} = \bar{\phi}_{t}
(A_0).$$

Denote $B^c =
\til M\sm B$ for every set $B$. We have the obvious relation 
\eq{eq: first split}{
\begin{split}
& \lambda\left(P_1\cap P_2(t)  \cap A^{(t)} \right) + \lambda\left(P_1^c\cap
P_2(t)  \cap A^{(t)}\right) \\
= & \lambda \left(P_2(t)  \cap A^{(t)}
\right) =\lambda \left(P_0  \cap A_0
\right)\geq
\frac{15}{16}\lambda\left(P_0\right).
\end{split}
}
We have 
$$\lambda \left(P_1^c\cap P_2(t) \cap A^{(t)} \right) \leq
\lambda \left(P_1^c\cap P_2(t) \right) =(1-b)\lambda(P_0),$$
so from \equ{eq: first split} and the definition of $b$: 
\[
\begin{split}
\lambda\left(P_1 \cap P_2(t) \cap A^{(t)} \right) & \geq \left(\frac{15}{16} -  1 +
b \right)\lambda\left(P_0\right) \\
& = \frac{1}{b}\left(\frac{15}{16}-1+b \right)\lambda(P_1 \cap P_2(t)). 
\end{split}
\] 
By the same reasoning,  
$$\lambda\left(P_1\cap P_2(t)  \cap A_{k} \right) \geq
\frac{1}{b}\left(\frac{15}{16} -  1 + b \right)\lambda\left(P_1 \cap P_2(t)
\right). $$ 
If, by contradiction, $A_k \cap A^{(t)}$ were of measure zero, from the
two preceding formulae we would have
\[
\begin{split}
& \frac{2}{b}\left(\frac{15}{16} -  1 + b\right)\lambda \left(P_1
\cap P_2(t) \right) \\
& \leq \lambda\left(P_1 \cap P_2(t) \cap A^{(t)} \right)+ \lambda \left
(P_1 \cap
P_2(t) \cap A_{k} \right)   \\  
& = \lambda\left(\left(P_1 \cap P_2(t) \cap A^{(t)} \right)\cup \left( P_1 \cap P_2(t)
\cap A_{k}\right) \right) \leq \lambda\left(P_1 \cap P_2(t)\right). 
\end{split}
\]
As $1/2 < b \leq 1$, the proportion $\frac{2}{b}(\frac{15}{16} -  1
+ b) > 1$ which leads to a contradiction. Therefore 
$ \lambda(A^{(t)} \cap A_{k}) >0$, as required. 
\qed
\section{Geodesic excursions and approximation by strips}
\begin{proof}[Proof of Theorem \ref{cor: Veech dichotomy}] We will
deduce the result from Theorem \ref{thm: main}. Using
Proposition \ref{prop: Zcovers correspondence}(iv) and passing to
finite-index subgroups, we can assume that the Veech groups 
of $M$ and $\til M = \til M_w$ are the same lattice $\Gamma$ in
$G$, and this lattice fixes $w$. As observed by Veech \cite{Veech -
alternative}, $\Gamma$ is non-uniform, and the cusps of $G/\Gamma$
correspond to cylinder decompositions 
of $M$. More precisely, cylinder decompositions in directions
$\theta_1, \theta_2$ on $M$ correspond to the same cusp in $G/\Gamma$
if and only if they are stabilized by conjugate parabolic elements $p_1, p_2
\in \Gamma$, where $p_i$ stabilizes direction $\theta_i$. Now suppose
$\til M$ has an infinite strip $\Sigma$ in direction $\theta_0$. It
follows from Proposition 
\ref{prop: having strips}, that to any infinite strip $\Sigma$ on
$\til M_w$, there is a corresponding cylinder $C$ on $M$ with core
curve $\delta$, such that
$k(\Sigma) = i(w, [[\delta]])$, $A(\Sigma) = \mathrm{area}(C)$ and
$v(\Sigma) = \hol(\delta)$. 

We say that strips $\Sigma_1, \Sigma_2$ are $\Gamma$-equivalent if
there is an affine automorphism of $\til M$ mapping $\Sigma_1$ to
$\Sigma_2$. It follows from Proposition \ref{prop:
Zcovers correspondence}(iv), and the fact that affine automorphisms
preserve the intersection pairing, that if
$\Sigma_1$ and $\Sigma_2$ are equivalent then 
$$k(\Sigma_1) =   
k(\Sigma_2), \ \ \ \ A(\Sigma_1)=A(\Sigma_2).$$
Also $v(\Sigma)$ is fixed by a parabolic element of $\Gamma$ so has a
discrete orbit under $\Gamma$. 

Let $\{x_1, \ldots, x_r\} \subset \R^2 \sm \{0\}$ be the vectors
$v(\Sigma_i)$, where $\Sigma_i$ range over representatives of the
$\Gamma$-equivalence classes of strips in $\til M$. By assumption $r
\geq 1$ and by the above discussion, $r$ is at most the
number of cusps in $\HH/\Gamma$. We obtain that
there is $d>0$ such that $\theta$ is well-approximated by strips if
and only if there is $i$ such that $\theta$ is $d$-well-approximated
by $\Gamma x_i$. 
Theorem \ref{cor: Veech dichotomy} is now a direct consequence of
Theorem \ref{thm: main} and 
Corollary \ref{cor: intermediate} and Proposition \ref{prop: a.e. and hd}.
\end{proof}

Let $\partial \, \HH = \R \cup \{0\}$ be the boundary of $\HH$ in the
one-point compactification of the complex plane.
Recall that there is a $\Gamma$-equivariant bijection $\mathrm{Vis}: S^1 \to \partial
\, \HH$ known as the `visibility map'. The {\em limit set} of a
Fuchsian group $\Gamma$ is the 
set of accumulation points of any of its orbits in $\HH \cup \partial
\, \HH$.

%

\begin{prop} \label{prop:Gdelta}
If $\til M \to M$ is a $\Z$-cover with an infinite strip, and such
that the Veech group $\Gamma$ of $\til M$ is of the first kind, then the set of well approximated directions is a dense $G_{\delta}$ subset of $S^1.$
\end{prop}

\begin{proof}
Let $\theta_0 \in
S^1$ be the direction of a strip $\Sigma \subset \til M$.  Let $A = A(\Sigma) >0, \, k =
k(\Sigma) \in \Z \sm \{0\}, \, x=v(\Sigma) \in \R^2$ be the
corresponding elements as in the preceding proof. Let $d < A/2$, and
enumerate the elements of $\Gamma$ as  $\{\gamma_1, \gamma_2,
\ldots\}$. For each $n$ let $G_n$ be the set of $\theta$ satisfying
\equ{eq: approximates} for some $\gamma \in \Gamma \sm \{\gamma_1,
\ldots, \gamma_n\}$. 
Then $G_n$ is clearly open. Let $\theta_1 \in G_n$ since $G_n$ contains all
but a finite subset of the orbit $\Gamma \theta_1$, it is dense
in $S^1$. 

Therefore the set $\Omega = \bigcap G_n$ is a dense $G_{\delta}$
subset of $\Lambda$. If $\theta$ belongs to $\Omega$ it is well approximable by an infinite number of strips directions.
\end{proof}

\begin{cor}\label{cor: FrUl} If $M$ is a lattice surface,
if $\til M \to M$ is a $\Z$-cover with an infinite strip, 
and   Veech group of $\til M$ is Fuchsian of the first kind, then the set
of ergodic directions on $\til M$ is a dense $G_{\delta}$. 
\end{cor}

\begin{proof}
By construction, the points in $\Omega$ are well approximated by
strips, namely the strips in the orbit of $\Sigma$, as in the
preceding proof. So in order to apply Theorem \ref{thm: main}, it
suffices to show that the directions in $\Omega$ are ergodic on finite covers of $M$.
This is certainly true since $M$ is a Veech surface. 
\end{proof}

\begin{remark} As we already mention,
As mentioned above, Corollary \ref{cor: FrUl} applies
 to examples studied by Fraczek and Ulcigrai \cite{Franczek-Ulcigrai}. They consider a genus 2
 square tiled surface $M$ and assume that the cocycle defining the $\Z$-cover
 belongs to the intersection of the absolute homology and the kernel
of the holonomy. By \cite{Zcovers}, Cor. 18, the Veech group of
the cover $\til{M}$ is of the first kind. By the result of Fraczek and Ulcigrai, the set of ergodic directions for the linear flow on $\til{M}$ has zero measure, but by Corollary \ref{cor: FrUl}, it is a dense $G_{\delta}$.

\end{remark}

\section{Examples}
There are several known constructions of examples satisfying the
assumptions of Theorem \ref{cor: Veech dichotomy}. 
The easiest is the infinite staircase, and there are other $\Z$-covers
of the torus which satisfy the hypotheses of Theorem \ref{cor: Veech
dichotomy}. In these examples the proof 
of ergodicity can be reduced to known results about skew products over an irrational
rotation. 

\begin{figure}[h!]
\begin{center}

\begin{tikzpicture}

 \tikzstyle{a line}=[very thick, dash pattern=on 8pt off 2pt];
\draw (0,0)   {}
        -- ++(45:1cm) node [below] {$e$}
        -- ++(45:1cm) 
                -- ++(90:1cm)
        -- ++(90:1cm) 
        -- ++(135:1cm) node [above] {$d$}
              -- ++(135:1cm)
        -- ++(180:1cm) node [above] {$c$}
                -- ++(180:1cm) 
        -- ++(225:1cm) node [left] {$e$}
                -- ++(225:1cm) 
        -- ++(270:1cm) node [left] {$a$}
                -- ++(270:1cm) 
       -- ++(315:1cm) node [below] {$d$}
              -- ++(315:1cm)
                 -- ++(360:1cm)  node [below] {$b$}
        -- (0,0);

\draw (4.82,0)   {}
        -- ++(45:1cm) node [below] {$f$}
        -- ++(45:1cm) 
                -- ++(90:1cm)  node [right] {$a$}
        -- ++(90:1cm) 
        -- ++(135:1cm)  node [above] {$g$}
              -- ++(135:1cm)
        -- ++(180:1cm)  node [above] {$b$}
                -- ++(180:1cm) 
        -- ++(225:1cm)  node [above] {$f$}
                -- ++(225:1cm) 
        -- ++(270:1cm) 
                -- ++(270:1cm) 
       -- ++(315:1cm)  node [below] {$g$}
              -- ++(315:1cm) 
               -- ++(360:1cm)  node [below] {$c$}
        -- (4.82,0);

         \draw[a line] (-3.41,1.41) -- ++(45:4.82cm);
         \draw[a line] (2.82,0) -- ++(45:4.82cm);
        
\end{tikzpicture}
\caption{\label{fig:octagon}
Double octagon.
}
\end{center}
\end{figure}
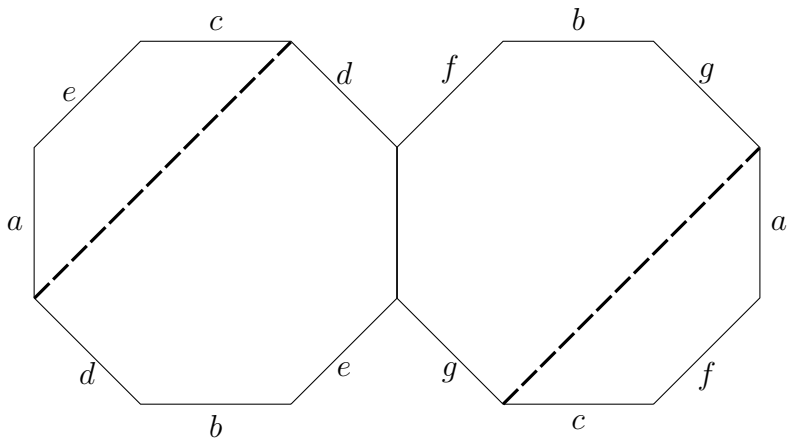

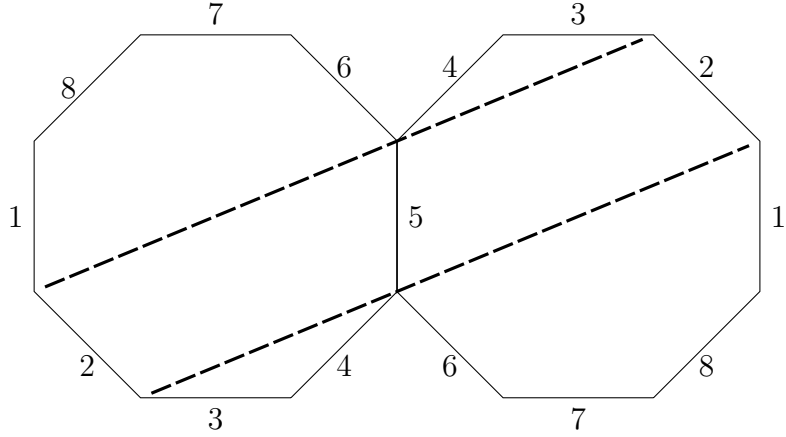
\begin{figure}[h!]
\begin{center}

\begin{tikzpicture}

 \tikzstyle{a line}=[very thick, dash pattern=on 8pt off 2pt];
\draw (0,0)   {}
        -- ++(45:1cm) node [below] {$4$}
        -- ++(45:1cm) 
                -- ++(90:1cm) node [right] {$5$}
        -- ++(90:1cm) 
        -- ++(135:1cm) node [above] {$6$}
              -- ++(135:1cm)
        -- ++(180:1cm) node [above] {$7$}
                -- ++(180:1cm) 
        -- ++(225:1cm) node [left] {$8$}
                -- ++(225:1cm) 
        -- ++(270:1cm) node [left] {$1$}
                -- ++(270:1cm) node (d1){}
       -- ++(315:1cm) node [below] {$2$}
              -- ++(315:1cm) node (b1){}
                 -- ++(360:1cm)  node [below] {$3$}
        -- (0,0);

\draw (4.82,0)   {}
        -- ++(45:1cm) node [below] {$8$}
        -- ++(45:1cm) 
                -- ++(90:1cm)  node [right] {$1$}
        -- ++(90:1cm) node(c2){}
        -- ++(135:1cm)  node [above] {$2$}
              -- ++(135:1cm) node (e2){}
        -- ++(180:1cm)  node [above] {$3$}
                -- ++(180:1cm) 
        -- ++(225:1cm)  node [above] {$4$}
                -- ++(225:1cm) 
        -- ++(270:1cm) 
                -- ++(270:1cm) 
       -- ++(315:1cm)  node [below] {$6$}
              -- ++(315:1cm) 
               -- ++(360:1cm)  node [below] {$7$}
        -- (4.82,0);

           \draw[a line] (b1) -- (c2);   \draw[a line] (d1) -- (e2);
        
\end{tikzpicture}
\caption{\label{fig:Veechoctagon}
Veech's double octagon.
}
\end{center}
\end{figure}

\subsection{Surfaces satisfying our conditions}
We now present some examples which cannot be reduced to skew products
over rotations, and to which our methods apply. 
An example of a $\Z$-cover which is a lattice surface is given at the
end of \cite{Zcovers}. 
The compact lattice surface $M$ is obtained from two octagons with pairs of sides
glued together (the gluings are different from the example of Veech in
\cite{Veech - alternative}). Hooper and Weiss proved that
one homology class $w_{0}$ ($(b) - (c)$ in Figure
\ref{fig:octagon}) induces a $\Z$-cover $\til M_{w_{0}}$ which is a
lattice surface. The direction of slope 1 has an infinite strip on
$\til M_{w_{0}}$ since the cylinder of slope 1 which intersects $(b)$
does not intersect $(c)$ in $M$. Thus a.e. direction on $\til
M_{w_0}$ is ergodic. 

\medskip

Many other examples can be constructed. We recall that Veech's original
examples \cite{Veech - alternative} of lattice surfaces are obtained
from two regular $n$-gons, where a 
side of one $n$-gon is glued to the opposite side of the other
$n$-gon. We will denote this translation surface by
$\mathrm{Reg}_{n}$ --- see figure \ref{fig:Veechoctagon} for the case $n =
8$. Assume that $n = 4m$, and denote by $1,2, \dots 4m$ the sides of
one of the $4m$-gons with counter clockwise orientation. Denote by $w$
the homology class of $1 + 3 + \dots + 4m-1$ and by $\mathrm{Reg}_{n,w}$ the
associated $\Z$-cover. We have $\hol(w) = 0$ since each segment appears
twice with opposite orientation. We recall that the Veech group of
$\mathrm{Reg}_{n}$  is generated by the horizontal Dehn twist and by the
rotation $r$ of angle $\pi/2m$. The horizontal Dehn twist lifts to
$\til{\mathrm{Reg}}_{n,w}$ since $w$ either intersects a horizontal cylinder once
with a positive orientation and once with a negative one or it does not
intersect it, see \cite[Proposition 24]{Zcovers}. The image of $w$ by
$r$ is 
the homology class $w*$  which is the sum of the segments $2, 4,
\dots, 4m$. Their union partitions $\mathrm{Reg}_n$ into two
connected components (the two $n$-gons), and $w = -w^*$, so by
\cite[Proposition 8]{Zcovers}, the rotation $r$ lifts to an affine
automorphism of $\til{ \mathrm{Reg}}_{n,w}$. In particular
$\til{\mathrm{Reg}}_{n,w}$ is a lattice surface.  
Moreover the direction of slope $\tan(\pi/n)$ has a strip since one
cylinder intersects the boundary of the $n$-gon along the sides $2,
2m+1$ (see figure \ref{fig:Veechoctagon}). 

A similar argument works for the surface $\mathrm{Reg}_{n= 2m+2}$. We
use the same class $w$ (sum of the odd-labelled segments), and in the
argument, exchange the roles of the
horizontal direction and the 
direction of slope  $\tan(\pi/n)$ (see figure
\ref{fig:Veechdecagon}). We leave the details to the reader.  

\ignore{
\begin{figure}[h!]
\begin{center}
\includegraphics[scale=0.7]{figure1_example.eps}
\caption{
An non-lattice parking garage with optimal dynamics.}
\label{parking}			
\end{center}
\end{figure}
}

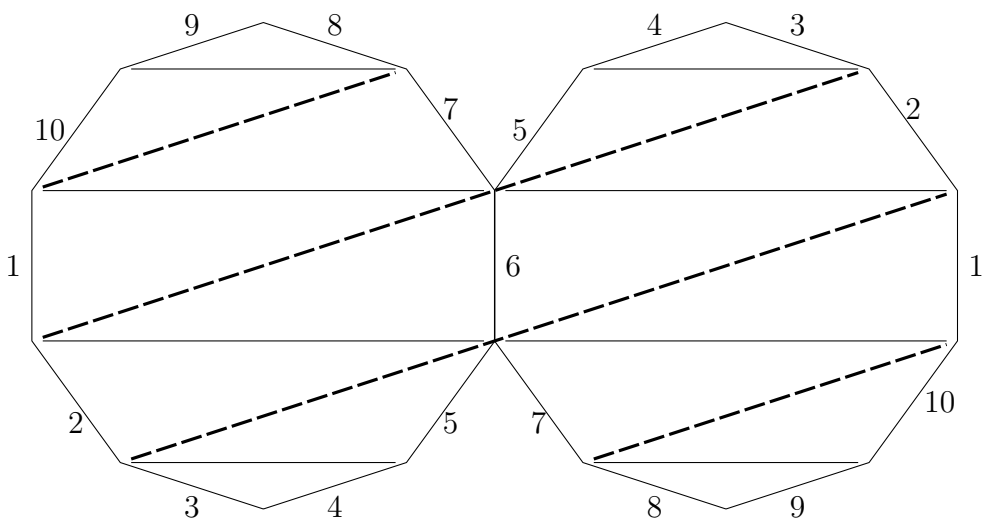
\begin{figure}[h!]
\begin{center}
\begin{tikzpicture}
 \tikzstyle{a line}=[very thick, dash pattern=on 8pt off 2pt];

 \tikzstyle{a line}=[very thick, dash pattern=on 8pt off 2pt];
\draw (0,0)   {}
 -- ++(18:1cm) node [below] {$4$}
        -- ++(18:1cm) node (a1){}
        -- ++(36+18:1cm) node [below] {$5$}
        -- ++(36+18:1cm) node (c1){}
                -- ++(72+18:1cm) 
        -- ++(72+18:1cm) node (e1){}
        -- ++(108+18:1cm) node [above] {$7$}
              -- ++(108+18:1cm) node (g1){}
        -- ++(144+18:1cm) node [above] {$8$}
                -- ++(144+18:1cm) 
        -- ++(180+18:1cm) node [above] {$9$}
                -- ++(180+18:1cm) node (h1){}
        -- ++(216+18:1cm) node [left] {$10$}
                -- ++(216+18:1cm) node (f1){}
       -- ++(252+18:1cm) node [left] {$1$}
              -- ++(252+18:1cm) node (d1){}
                   -- ++(288+18:1cm) node [below] {$2$}
              -- ++(288+18:1cm) node (b1){}
                   -- ++(324+18:1cm) node [below] {$3$}
        -- (0,0);
        
              \draw (a1)--(b1);    \draw (c1)--(d1);   
               \draw (e1)--(f1);   
 \draw (g1)--(h1);

\draw (6.15,0)   {}
 -- ++(18:1cm) node [below] {$9$}
        -- ++(18:1cm) node (a2){}
        -- ++(36+18:1cm) node [right] {$10$}
        -- ++(36+18:1cm) node (c2){}
                -- ++(72+18:1cm) node [right] {$1$}
        -- ++(72+18:1cm) node (e2){}
        -- ++(108+18:1cm) node [above] {$2$}
              -- ++(108+18:1cm) node (g2){}
        -- ++(144+18:1cm) node [above] {$3$}
                -- ++(144+18:1cm) 
        -- ++(180+18:1cm) node [above] {$4$}
                -- ++(180+18:1cm) node (h2){}
        -- ++(216+18:1cm) node [left] {$5$}
                -- ++(216+18:1cm) node (f2){}
       -- ++(252+18:1cm) node [right] {$6$}
              -- ++(252+18:1cm) node (d2){}
                   -- ++(288+18:1cm) node [below] {$7$}
              -- ++(288+18:1cm) node (b2){}
                   -- ++(324+18:1cm) node [below] {$8$} 
                   --        (6.15,0);
                   
   \draw (a2)--(b2);    \draw (c2)--(d2);   
               \draw (e2)--(f2);   
 \draw (g2)--(h2);

                    \draw[a line] (b1) -- (e2);
                        \draw[a line] (d1) -- (g2);
                          \draw[a line] (f1) -- (g1);
   \draw[a line] (b2) -- (c2);

\end{tikzpicture}

\caption{\label{fig:Veechdecagon}
Veech's double decagon.
}
\end{center}
\end{figure}

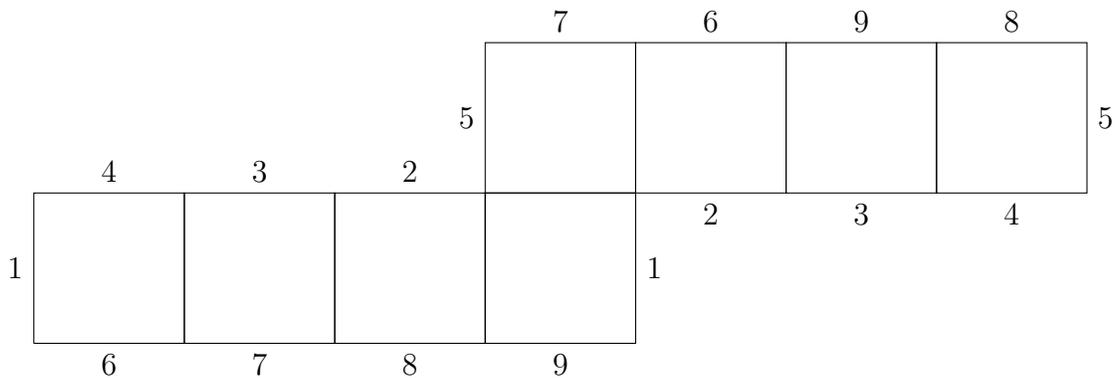
\begin{figure}[h!]
\begin{center}

\begin{tikzpicture}

\draw (0,0)   
         -- (1,0) node [below] {$6$}
        -- (2,0) --(2,1) 
               -- (2,2)  
        -- (1,2) node [above] {$4$}
        --(0,2) 
         --(0,1) node [left] {$1$}
          -- (0,0);
          
 \draw (2,0)   
         -- (3,0) node [below] {$7$}
        -- (4,0) --(4,1)
               -- (4,2)  
        -- (3,2) node [above] {$3$}
        --(2,2) 
         --(2,1)
          -- (2,0);
         
  \draw (4,0)   
         -- (5,0) node [below] {$8$}
        -- (6,0) --(6,1)
               -- (6,2)  
        -- (5,2) node [above] {$2$}
        --(4,2) 
         --(4,1)
          -- (4,0);      
          
  \draw (6,0)   
         -- (7,0) node [below] {$9$}
        -- (8,0) --(8,1) node [right]{$1$}
               -- (8,2)  
        -- (7,2) 
        --(6,2) 
         --(6,1) 
          -- (6,0);          
             
   \draw (6,2)   
         -- (7,2)  
        -- (8,2)  --(8,3)  
               -- (8,4)  
        -- (7,4) node [above] {$7$}
        --(6,4) 
         --(6,3)  node [left] {$5$}
          -- (6,2);     
          
     \draw (8,2)   
         -- (9,2) node [below] {$2$}
        -- (10,2) --(10,3)
               -- (10,4)  
        -- (9,4) node [above] {$6$}
        --(8,4) 
         --(8,3) 
          -- (8,2);           
          
           \draw (10,2)   
         -- (11,2) node [below] {$3$}
        -- (12,2) --(12,3)
               -- (12,4)  
        -- (11,4)  node [above] {$9$}
        --(10,4) 
         --(10,3)
          -- (10,2);          
          
               \draw (12,2)   
         -- (13,2) node [below] {$4$}
        -- (14,2) --(14,3) node [right] {$5$}
               -- (14,4)  
        -- (13,4) node [above] {$8$}
        --(12,4) 
         --(12,3) 
          -- (12,2);

\end{tikzpicture}

\caption{\label{fig:wollmichsau}
Wollmichsau.
}
\end{center}
\end{figure}
\subsection{An extension of Kesten's theorem}
Kesten \cite{Kesten} showed that if $\theta$ is an irrational
direction on the infinite staircase, then there is no $x$ for which 
$\{\overline{\phi}_tx: t \in \R\}$ is bounded. Adapting an argument of Petersen
\cite{Peterson}, we show the following: 

\begin{prop}\name{prop: Kesten}
Let $\til M \to M$ be a recurrent $\Z$-cover and let $\theta$ be a
direction which is ergodic on $\til M$. Then there
is no bounded trajectory on $\til M$.  
\end{prop}
\begin{proof} 
Suppose by contradiction that $\til x_1 \in \til M$ is a point with a
bounded trajectory. Since $\theta$ is an ergodic direction on $\til
M$, there are points in $\til M$ whose trajectory visits every subset of $\til M$ of
positive measure; suppose $\til x_2$ is such a point. 
Let $x_1, x_2$ denote the projections of $\til x_1, \til x_2$ in
$M$. Then the choice of $\til x_2$ implies that for the cocycle
$\alpha$ defined by \equ{eq: cocycle} we have 
\eq{eq: for a contradiction}{
\{\alpha(x_2, t): t>0\} = \Z.
}
Since the orbit of $\til x_1$ is bounded there is some $N$ such that
$|\alpha(x_1,t)| \leq N$ for all $t>0$. 
Let $\{\beta_{x_0, x}\}$ be the family of curves
described in \S \ref{section: essential}, let $\til w$ be a smooth
representative of the homology class $w$, and fix $T \in \R$.
There  
is a neighborhood $U$ of $x_2$ in $M$ such that 
$\phi_T|_U$ is an isometry, and such that the number of essential
intersections of the curve $s \mapsto \phi_sx$ with $\til w$, for $s$
between $0$ and $T$, is the same for all $x \in U$. This implies via
\equ{eq: cocycle} that 
for all $x_3 \in U$, $|\alpha(x_2, T) - \alpha(x_3, T)| \leq K$, where
$K$ is a constant depending only on the choice of paths
$\{\beta_{x_0,x}\}.$
Since $\theta$ is ergodic on $\til M$, it is also ergodic on $M$, and
since $M$ is compact, this implies that the straightline flow in
direction $\theta$ on $M$ is minimal. Thus there is $t_0>0$ such that
$T+t_0>0$ and 
$x_3 = \phi_{t_0}(x_1) \in U,$ 
and \equ{eq: cocycle identity} implies
$$
|\alpha(x_3, T)| = |\alpha(\phi_{t_0}(x_1), T)| =
|\alpha(x_1,T+t_0) -
\alpha(x_1, t_0)| \leq 2N.
$$
Therefore $|\alpha(x_2, T)| \leq 2N+K$, contradicting \equ{eq: for a contradiction}.
\end{proof} 

\subsection{An interesting non-example} The existence of a strip is
not automatic. On the Wollmilchsau $W$, depicted on figure 
\ref{fig:wollmichsau} there is a relative homology class $w_{1}$ of
zero holonomy such that {\it every} cylinder in  $W$ lifts to
$\til{W}_{w_{1}}$ as a union of cylinders. Indeed one can check that the
class $w_1$ which is the difference of the segments marked $2$ and $4$
on figure  
\ref{fig:wollmichsau} has this property. Thus, there is no infinite
strip on $\til{W}_{w_{1}}$ and our method does not apply to this
case. On the other hand, by
\cite[Proposition 30]{Zcovers}, $\til{W}_{w_1}$ is a lattice surface. 
In \cite{AHM} it is shown, using a different method, that every
irrational direction in $\til{W}_{w_1}$ is ergodic. 

In a certain sense, it is not easy to construct recurrent $\Z$-covers
without infinite strips. We use the following observation:
\begin{prop}
\name{prop: observation strips}
If the cylinder core curves of $M$ generate a finite-index subgroup of
$H_1(M; \Z)$ then for any
holonomy-free $w \in
H_1(M;\Z)$, the recurrent $\Z$-cover $\til M_w$ has infinite strips. 

\end{prop}

\begin{proof}
This is immediate from the non-degeneracy of the intersection
pairing. 
\end{proof}
The condition that the cylinder core curves generate $H_1(M;\Z)$ was
studied in \cite{Monteil}, where it is shown that it holds for almost
every surface $M$. 

\medskip 

{ \bf Acknowledgments:} 
This research was supported by the Israel-France mathematics
program 07 MATH  F 5 `Geometry and Dynamics of flat surfaces', by
Israel Science Foundation grant 190/08,  by the ANR Teichm\"uller `projet
blanc' ANR-06-BLAN-0038 and by ANR `projet blanc' GeoDyM. 

We thank Manfred Eisiedler, Pat Hooper and Thierry Monteil for enlightening discussions,
David Ralston for telling us about Kesten's theorem, 
and Anton Zorich for telling us about Panov's work. \\


\begin{thebibliography}{DRV2}
\bibitem[ANSS]{ANSS} J. Aaronson, H. Nakada, O. Sarig and R. Solomyak,
{\em Invariant measures and asymptotics for some skew products},
Isr. J. Math. {\bf 128} (2002), 93--134.

\bibitem[AvHuMa]{AHM} A. Avila, P. Hubert and C. Mattheus, {\em in preparation.}

\bibitem[Co]{Co} J.P. Conze, {\em Equir\'epartition et ergodicit\'e de
transformations cylindriques}, s\'eminaire de probabilit\'e de Rennes
(1976), 1--21. 

\bibitem[CoFr]{Conze-new} J.-P. Conze and K. Fraczek, {\em Cocycles
over interval exchange transformations and multivalued Hamiltonian
flows}, preprint (2010) {\tt http://arxiv.org/abs/1003.1808}  


\bibitem[Ei]{Einsiedler} M. Einsiedler, personal communication. 

\bibitem[EiKaPo]{EKP} M. Einsiedler, S. Kadyrov and A. Pohl, {\em
  Escape of mass and entropy for diagonal flows in real rank one situations},
  preprint (2011) {\tt http://arxiv.org/abs/1110.0910} 

\bibitem[FrUl]{Franczek-Ulcigrai} K. Fraczek and C. Ulcigrai, {\em
  Non-ergodic $\Z$-periodic billiards and infinite translation
  surfaces}, preprint (2011), {\tt 
http://arxiv.org/abs/1109.4584}


\bibitem[Ho1]{Ho} P. Hooper, {\em Dynamics on an infinite staircase
with the lattice 
 property}, preprint (2008) {\tt
http://arxiv.org/abs/0802.0189}

\bibitem[Ho2]{Ho-classification} P. Hooper, {\em The invariant
measures of some infinite interval exchange maps}, preprint (2010)
{\tt
http://arxiv.org/abs/1005.1902}

\bibitem[HoHuWe]{HuWe} W. P. Hooper, P. Hubert and B. Weiss, {\em Dynamics on the
infinite staircase surface}, preprint 
 (2008) {\tt http://www.math.bgu.ac.il/~barakw/staircase.pdf}

\bibitem[HoWe]{Zcovers} W. P. Hooper and B. Weiss, {\em Generalized
staircases: recurrence and symmetry} (2009) to appear in
Ann. Inst. Fourier. 

 
\bibitem[HuSc]{HuGabi} P. Hubert and G. Schmithuesen, {\em Infinite translation surfaces with
 infinitely generated Veech group}, J. Mod. Dyn. 4 (2010), no. 4, 715--732. 



\bibitem[Ka]{Ka} S. Katok, {\bf Fuchsian groups}, Chicago Lectures in
Mathematics. University of Chicago Press (1992). 

\bibitem[Ke]{Kesten} H. Kesten, {\em On a conjecture of Erd\H{o}s and
  Sz\"usz related to uniform distribution mod 1, }
Acta Arith. {\bf 12} 1966/1967 193--212. 

\bibitem[KlWe]{margulis} D. Kleinbock and B. Weiss, {\em Modified
  Schmidt games and a conjecture of Margulis}, preprint (2011). 

\bibitem[KS]{KS} D. K\"onig and A. Sz\"ucs, {\em Mouvement d'un point
abandonn\'e \`a l'interieur d'un cube}, Rend. del circulo matematico
di Palermo {\bf 36} (1913) 79--90. 

\bibitem[Ma]{Masur-Duke} H. Masur, {\em Hausdorff dimension of the set
of nonergodic foliations of a quadratic differential}, Duke
Math. J. {\bf 66} (1992) 387--442.

\bibitem[MaTa]{MaTa} H. Masur, S. Tabachnikov, {\em Rational billiards
and flat structures}, in {\bf Handbook of dynamical systems}, Vol. 1A,
North-Holland, Amsterdam (2002), pp.~1015--1089. 


\bibitem[Mo]{Monteil} T. Monteil, {\em A homological condition for a
dynamical and illuminatory classification of torus branched
coverings,} preprint (2006) {\tt http://arxiv.org/abs/math/0603352 }


\bibitem[Pa]{Pa} D. Panov, {\em Foliations with unbounded deviation on
$\mathbb{T}^2$}, J. Mod. Dyn.  {\bf 3}  (2009),  no. 4, 589--594. 

\bibitem[P]{Patterson} S. J. Patterson, {\em Diophantine approximation
  in Fuchsian groups}, Phil. Trans. Roy. Soc. Lond.
{\bf 282} (1976) 527--563.

\bibitem[Pe]{Peterson} K. Petersen, {\em On a series of cosecants
  related to a problem in ergodic theory, }  Compositio Math.  {\bf 26}  (1973) 313--317.


\bibitem[Ra]{Ralston} D. Ralston, {\em $1/2$-heavy sequences driven by rotations,} (2011), preprint {\tt http://arxiv.org/abs/1106.0577}. 
\bibitem[Sch]{Schmidt} K. Schmidt, {\bf Cocycles of
ergodic transformation groups},
MacMillan (India) (1977), available at {\tt
http://www.mat.univie.ac.at/\~{}kschmidt}


\bibitem[V]
{Veech - alternative} W. A. Veech, {\em
Teichm\"uller curves in moduli space, Eisenstein series and an
application to triangular billiards}, 
Invent. Math. {\bf 97} (1989), no. 3, 553--583.

\bibitem[Vi]{Viana} M. Viana, {\em 
Dynamics of interval exchange maps and Teichm\"uller flows}
Lecture notes of graduate courses taught at IMPA in 2005 and 2007.
Working preliminary manuscript. 
{\tt http://w3.impa.br/~viana/out/ietf.pdf}


\bibitem[Zo]{Zo} A. Zorich, {\em Flat surfaces},  Frontiers in number
theory, physics, and geometry I, Springer, Berlin, (2006), pp.~
437--583. 


\end{thebibliography}
\end{document}